\documentclass[11pt,a4paper,reqno]{amsart}

\usepackage[T1]{fontenc}
\usepackage[utf8]{inputenc}
\usepackage[left=2.5cm,right=2.5cm,top=2.7cm,bottom=3cm]{geometry}
\usepackage{amsmath,amssymb,amsthm}
\usepackage{bbm}
\usepackage{microtype}
\usepackage[hidelinks]{hyperref}

\newtheorem{theorem}{Theorem}
\newtheorem{lemma}{Lemma}
\newtheorem{prop}{Proposition}

\DeclareMathOperator{\supp}{supp}
\DeclareMathOperator{\dist}{dist}

\title{Stability of concentration in the Paley--Wiener space}
\author{Denis Zelent}
\address{Department of Mathematical Sciences,
Norwegian University of Science and Technology (NTNU),
7491 Trondheim, Norway}
\email{denis.zelent@ntnu.no}
\thanks{The author was supported by Grant 334466 of the
Research Council of Norway.}
\date{}

\begin{document}

\begin{abstract}
We prove stability of concentration in the Paley--Wiener
space for each fixed time--bandwidth product. The concentration
deficit controls the squared $L^2(\mathbb R)$-distance from
a translated first prolate spheroidal wave function and
the squared normalized symmetric difference from an interval
at the same center. 
The obtained exponent for the function distance is optimal.
The main idea is the introduction of the set discrepancy weighted by the best optimizers in place of the symmetric difference.
The proof then combines elementary Hilbert space arguments with classical properties of the prolate spheroidal wave functions.
\end{abstract}

\maketitle

\section{Introduction}

\subsection{The concentration problem}
\label{sec:main-results}
Let
\[
PW=\left\{f\in L^2(\mathbb R):
\supp\widehat f\subset[-1/2,1/2]\right\},
\qquad
\widehat f(\xi)=\int_{\mathbb R}f(x)e^{-2\pi i x\xi}\,dx,
\]
denote the Paley--Wiener space and the Fourier transform accordingly. Given a measurable set $E\subset\mathbb R$ with
$0<|E|<\infty$, we consider the fraction
\[
\frac{\displaystyle\int_E|f(x)|^2\,dx}{\|f\|_2^2},
\qquad 0\ne f\in PW.
\]
This measures how much of a band-limited function is concentrated
on $E$. Finding an upper bound in terms of $|E|$, and identifying
the functions and sets that attain it, is a classical problem in
Fourier analysis. It is closely connected with uncertainty
principles and the recovery of signals from incomplete data;
see \cite{DS89,Slepian}.

When $E$ is an interval, the concentration problem is solved
by the classical theory of prolate spheroidal wave functions,
developed in the work of Slepian, Landau and Pollak
\cite{LP,SP}. For $s>0$, let
\[
I_s=(-s/2,s/2),
\qquad
\lambda_0(s)=
\max_{0\ne f\in PW}
\frac{\displaystyle\int_{I_s}|f(x)|^2\,dx}{\|f\|_2^2}.
\]
The maximum is attained precisely by the nonzero scalar multiples
of the first prolate spheroidal wave function, scaled to the
interval $I_s$ and the frequency band $[-1/2,1/2]$;
see \cite[Sections III, IV.3 and V]{SP}.
We denote it by $\psi_s$ in the case when it is real, even, and normalized by $\|\psi_s\|_{L^2(\mathbb R)}=1$ and $\psi_s(0)>0$.
It satisfies, for any $x\in\mathbb R$,
\begin{equation}
\int_{-s/2}^{s/2}
\frac{\sin\pi(x-y)}{\pi(x-y)}\psi_s(y)\,dy
=
\lambda_0(s)\psi_s(x).
\label{eq:prolate-integral}
\end{equation}
After the change of variables $x=st/2$, the function
$t\mapsto\psi_s(st/2)$ is proportional to the angular prolate
spheroidal wave function of order zero and degree zero, with
parameter $c=\pi s/2$; see \cite[Section V]{SP}.
The largest concentration eigenvalue is simple, and
\[
\int_{I_s}|\psi_s(x)|^2\,dx=\lambda_0(s),
\qquad 0<\lambda_0(s)<1.
\]

The corresponding problem for a general set $E$ was considered
by Donoho and Stark \cite{DS89,DS93}. Abreu and Speckbacher
\cite[Theorem 1.1]{AS} proved that intervals are optimal among
all measurable sets of the same measure. In our normalization,
\begin{equation}
\int_E|f(x)|^2\,dx
\le\lambda_0(s)\|f\|_2^2,
\qquad f\in PW,\quad |E|=s.
\label{eq:interval}
\end{equation}
The pair $(\psi_s(\cdot-a),I_s+a)$ attains equality for every
$a\in\mathbb R$.

We also use the equality characterization
\begin{equation}
\int_E|f(x)|^2\,dx=\lambda_0(s)\|f\|_2^2
\quad\Longleftrightarrow\quad
\begin{cases}
E=I_s+a \quad\text{up to a null set},\\
f=c\psi_s(\cdot-a),
\end{cases}
\label{eq:equality}
\end{equation}
for $0\ne f\in PW$, $|E|=s$, and some $a\in\mathbb R$ and
$c\in\mathbb C\setminus\{0\}$.
The set assertion can be obtained by a real-line adaptation of the
zero-block variation used for the circle interval theorem in
\cite[Sections 4--8]{AS}; that adaptation is not reproduced here.
Once the set is an interval, the function assertion follows from
simplicity of the first prolate eigenvalue \cite[Section III]{SP}.
%We are aware that the authors of \cite{AS} work on a changed version of their manuscrpit, hence we don't provide the full proof in case they will do so. If that won't happen, we will include the proof here.

We study the stability of \eqref{eq:interval}: when the
concentration is close to $\lambda_0(s)$, must the function be
close to a translated prolate spheroidal wave function and the set close to the interval at the same center?

Fix $s>0$ for the remainder of the paper. For $\|f\|_2=1$ and
$|E|=s$, define the concentration deficit by
\begin{equation}
\Delta(f,E)=\lambda_0(s)-\int_E|f(x)|^2\,dx.
\label{eq:deficit}
\end{equation}
For a given $a\in\mathbb{R}$, write
\[
d_a(f)=\inf_{\theta\in\mathbb R}
\|f-e^{i\theta}\psi_s(\cdot-a)\|_2
\]
to denote the distance from $f$ to the closest prolate spheroidal wave function.
For the set, the usual normalized distance from an interval is
\[
\mathcal A(E)
=
\inf_{a\in\mathbb R}
\frac{|E\triangle(I_s+a)|}{s}.
\]
This is the one-dimensional Fraenkel asymmetry. It was introduced
for the quantitative isoperimetric inequality \cite{FMP,Hall} and is equally natural for the Faber--Krahn inequality \cite{BDV}:
perimeter, volume and the first Dirichlet eigenvalue do not depend
on where a detached piece of the set is placed, and the asymmetry
stops depending on it as soon as the piece is far enough from the rest of the set. Concentration, by contrast, does depend on where the piece goes, and this is what the asymmetry misses: it records how much of the set is misplaced, not how far.

To illustrate this distinction, let $r=s/2$ and, for
$0<\varepsilon<r$ and $b\ge\varepsilon$, put
\[
E_{\varepsilon,b}
=
(-r,r-\varepsilon)
\cup
(r-\varepsilon+b,r+b).
\]
This is the interval $I_s$ with a piece of length $\varepsilon$
removed from its right end and shifted to the right by $b$.
For every $b\ge\varepsilon$,
\[
|E_{\varepsilon,b}|=s,
\qquad
\mathcal A(E_{\varepsilon,b})=\frac{2\varepsilon}{s}.
\]

The loss of optimal concentration, however, distinguishes
these sets. Indeed, write
\[
\eta(E)=\lambda_0(s)-\max_{\substack{g\in PW\\\|g\|_2=1}}\int_E|g|^2.
\]
Thus $\eta(E)$ is the smallest concentration deficit on $E$. One can then compute that, as $\varepsilon\downarrow0$,
\[
\eta(E_{\varepsilon,\varepsilon})\asymp_s\varepsilon^2,
\qquad
\eta(E_{\varepsilon,b_0})\asymp_{s,b_0}\varepsilon
\]
for every fixed $b_0>0$.
The two families therefore have the same Fraenkel asymmetry
but different orders of concentration deficit. In particular,
no fixed power of $\mathcal A(E)$ is uniformly comparable
to $\eta(E)$ over sets of measure $s$. 
This motivates the introduction of set discrepancy weighted by the best optimizers,
\begin{equation}
D_a(E)=\lambda_0(s)-\int_E|\psi_s(x-a)|^2\,dx,
\qquad
\mathcal D(E)=\inf_{a\in\mathbb R}D_a(E).
\label{eq:set-discrepancy}
\end{equation}
We call any minimizing point $a_E$ an optimal center of $E$.
% Such a point exists; this is checked below.
% The advantage of this discrepancy for our argument is that, as we will see, it is itself a concentration deficit on $E$.
% $D_a(E)=\Delta(\psi_s(\cdot-a),E)$
% It also has a weighted
% symmetric-difference representation, given in
% \eqref{eq:weighted}, which allows us to cancel the boundary
% values of a density perturbation. The usual symmetric-difference
% estimate will follow from this weighted estimate.

\subsection{Motivation behind the new notion of set discrepancy - concentration in the Fock space}\label{STFT} The idea and usefulness of using such a measure of discrepancy came to the author first in the setting of the Fock space, while we realized later that analogous ideas work well in the Paley--Wiener case as well. 
Let $\varphi(t)=2^{1/4}e^{-\pi t^2}$ and let
\[
V_\varphi f(x,\xi)=\int_{\mathbb R}f(t)\,\overline{\varphi(t-x)}\,
e^{-2\pi i\xi t}\,dt,
\qquad z=(x,\xi)\in\mathbb R^2,
\]
be the short-time Fourier transform with Gaussian window; see
\cite{Gro} for background. For $z_0=(x_0,\xi_0)$ write
$\pi(z_0)f(t)=e^{2\pi i\xi_0 t}f(t-x_0)$, so that
$|V_\varphi(\pi(z_0)\varphi)(z)|^2=e^{-\pi|z-z_0|^2}$.
For a measurable $U\subset\mathbb C$ with $|U|=s$, the localization
operator $\mathcal L_U$ of Daubechies \cite{Dau} is defined by
$\langle\mathcal L_Uf,g\rangle=\int_U V_\varphi f\,\overline{V_\varphi g}\,dA$,
where $dA$ is planar Lebesgue measure; it is positive and compact, and
$\langle\mathcal L_Uf,f\rangle$ is the part of
$\|f\|_2^2=\|V_\varphi f\|_2^2$ concentrated on $U$. Nicola and Tilli
\cite{NT} proved that $\langle\mathcal L_Uf,f\rangle\le(1-e^{-s})\|f\|_2^2$,
with equality only for disks and shifted Gaussians, and G\'omez,
Guerra, Ramos and Tilli \cite{GGRT} proved a sharp quantitative
version of this inequality.
% Let $\varphi(t)=2^{1/4}e^{-\pi t^2}$ be the normalized Gaussian,
% so that
% \[
% |V_\varphi(\pi(z_0)\varphi)(z)|^2
% =
% e^{-\pi|z-z_0|^2},
% \]
% where $\pi(z_0)$ denotes a time--frequency shift.

To define discrepancy weighted by the best optimizers in this setting, for a measurable set $U\subset\mathbb C$ with $|U|=s>0$ we put
\[
\mathcal D(U)
=
\frac1s\left(
1-e^{-s}
-
\sup_{z_0\in\mathbb C}
\int_U e^{-\pi|z-z_0|^2}\,dA(z)
\right),
\]
and
\[
\mathcal A(U)
=
\inf_{z_0\in\mathbb C}
\frac{|U\triangle D(z_0,\sqrt{s/\pi})|}{s}.
\]
Here $D(z_0,r)$ is the disk of radius $r$ centered at $z_0$.

Fix $z_0\in\mathbb C$, let $B=D(z_0,\sqrt{s/\pi})$, and write
\[
m=|B\setminus U|=|U\setminus B|
=\frac12|U\triangle B|.
\]
For a given $m$, the loss of Gaussian mass is smallest when
the missing part lies just inside the disk boundary and the
added part just outside it. Comparing with these adjacent
annuli and using the radial variable $t=\pi|z-z_0|^2$, we get
\begin{align*}
\int_{B\setminus U}e^{-\pi|z-z_0|^2}\,dA(z)
&\ge
\int_{s-m}^{s}e^{-t}\,dt
=e^{-s}(e^m-1),\\
\int_{U\setminus B}e^{-\pi|z-z_0|^2}\,dA(z)
&\le
\int_s^{s+m}e^{-t}\,dt
=e^{-s}(1-e^{-m}).
\end{align*}
Subtracting and using $\cosh m-1\ge m^2/2$, we obtain
\[
e^{-s}m^2
\le
2e^{-s}(\cosh m-1)
\le
(1-e^{-s})-\int_Ue^{-\pi|z-z_0|^2}\,dA(z)
\le m.
\]
The upper bound follows by discarding the added mass and
bounding the Gaussian by one on the missing part.
Dividing by $s$, substituting $m=|U\triangle B|/2$, and
taking the infimum over $z_0$ gives
\begin{equation}
\frac{se^{-s}}4\,\mathcal A(U)^2
\le
\mathcal D(U)
\le
\frac12\,\mathcal A(U).
\label{eq:Gaussian-geometric-comparison}
\end{equation}
% Since the Gaussian weight is radially decreasing, the two
% halves of the symmetric difference can be rearranged into
% annuli adjacent to the disk boundary. Using
% \[
% \int_0^r 2\pi t\,e^{-\pi t^2}\,dt=1-e^{-\pi r^2},
% \]
% we obtain
% \begin{equation}
% \frac{se^{-s}}4\,\mathcal A(U)^2
% \le
% \mathcal D(U)
% \le
% \frac12\,\mathcal A(U).
% \label{eq:Gaussian-geometric-comparison}
% \end{equation}
% Both powers occur: moving a thin annulus across the boundary
% gives $\mathcal D\asymp_s\mathcal A^2$, whereas removing a
% small disk from the center and placing it far away gives
% $\mathcal D\asymp_s\mathcal A$.

Now let $\lambda_0(U)$ be the largest eigenvalue of the
Gaussian localization operator $\mathcal L_U$. By \cite{NT},
$\lambda_0(U)\le1-e^{-s}$. Define the relative spectral deficit
by
\[
\varepsilon
=
1-\frac{\lambda_0(U)}{1-e^{-s}}.
\]
The functional stability estimate
\cite[Theorem 1.1, inequality (1.7)]{GGRT}, applied to a
normalized top eigenfunction $f_U$, gives a shifted Gaussian
$g=e^{i\theta}\pi(z_0)\varphi$ such that
\[
\|f_U-g\|_2^2\le Ce^s\varepsilon.
\]
Write $g=af_U+h$, with $h\perp f_U$. Then
\[
1-|a|^2=\|h\|_2^2\le\|g-f_U\|_2^2,
\]
and positivity together with the eigenvalue equation gives
\[
\langle\mathcal L_Ug,g\rangle
=
|a|^2\lambda_0(U)+\langle\mathcal L_Uh,h\rangle
\ge
|a|^2\lambda_0(U).
\]
Consequently,
\begin{align*}
s\mathcal D(U)
&\le
(1-e^{-s})-\langle\mathcal L_Ug,g\rangle\\
&\le
(1-e^{-s})-\lambda_0(U)
+\lambda_0(U)\|g-f_U\|_2^2\\
&\le
(1-e^{-s})\bigl(\varepsilon+\|g-f_U\|_2^2\bigr)\\
&\le
(1-e^{-s})(1+Ce^s)\varepsilon.
\end{align*}

Conversely, every shifted Gaussian is an admissible normalized
function in the concentration problem, so
\[
\sup_{z_0\in\mathbb C}
\int_U e^{-\pi|z-z_0|^2}\,dA(z)
\le\lambda_0(U).
\]
Hence
\[
s\mathcal D(U)
\ge
(1-e^{-s})-\lambda_0(U)
=
(1-e^{-s})\varepsilon.
\]
Combining the two estimates, we obtain
\begin{equation}
\frac{1-e^{-s}}s\,\varepsilon
\le
\mathcal D(U)
\le
\frac{1-e^{-s}}s(1+Ce^s)\,\varepsilon.
\label{eq:Gaussian-spectral-comparison}
\end{equation}
Thus, for fixed $s$, the weighted discrepancy is comparable
to the relative spectral deficit. From
\eqref{eq:Gaussian-geometric-comparison} it follows that
\[
\mathcal A(U)\le \frac{2}{s}\sqrt{(e^s-1)(1+Ce^s)}\sqrt{\varepsilon},
\]
which recovers, for the spectral deficit, the set estimate of \cite{GGRT}.
The improved power concerns the weighted discrepancy:
the square-root exponent is optimal both for Fraenkel
asymmetry and for the distance of normalized first
eigenfunctions from the Gaussian family; see
\cite[Theorem 4]{Ramos}.

We now return to the Paley--Wiener space and the discrepancy
$\mathcal D(E)$ defined in \eqref{eq:set-discrepancy}.
The theorem below shows that this discrepancy and the
squared function distance are controlled by the concentration
deficit at the same center.

\subsection{The main result}
\begin{theorem}[Global stability]
\label{thm:main}
For every $s>0$ there exists a constant $C_s>0$ such that
every measurable $E\subset\mathbb R$ with $|E|=s$, every
$f\in PW$ with $\|f\|_2=1$, and every optimal center $a_E$
satisfy
\begin{equation}
\mathcal D(E)+d_{a_E}(f)^2\le C_s\Delta(f,E).
\label{eq:main}
\end{equation}
Moreover, at the same center,
\begin{equation}
\left(\frac{|E\triangle(I_s+a_E)|}{s}\right)^2 + d_{a_E}(f)^2
\le C_s\Delta(f,E),
\label{eq:main-measure}
\end{equation}
after increasing $C_s$ if necessary.
\end{theorem}

In particular, both $d_{a_E}(f)$ and $\mathcal A(E)$ are bounded
by $C_s\Delta(f,E)^{1/2}$. The assertion with a common center
is stronger than estimating the two distances with independently
chosen centers. No connectedness, boundedness, or boundary
regularity is assumed for $E$. The proof uses the equality characterization \eqref{eq:equality}. 
The exponent of the function estimate cannot be improved; see
Section~\ref{sec:sharpness}.

Theorem~\ref{thm:main} is the Paley--Wiener counterpart of two
recent stability results. The first is that of G\'omez, Guerra,
Ramos and Tilli \cite{GGRT} for the short-time Fourier transform,
discussed in Section~\ref{STFT}. The second is that of
Garc\'ia-Ferrero and Ortega-Cerd\`a \cite{GO} for the
finite-dimensional analogue of \eqref{eq:interval}, in which
band-limited functions are replaced by polynomials of bounded
degree; they prove a local stability estimate and a global one
through a Wehrl-type entropy. In \cite{GO2} they extend both to
homogeneous polynomials in several variables and recover the stability results in the Bargmann--Fock space. 
Our passage from local to global stability uses compactness, so the present proof
does not provide an explicit value of $C_s$.

For a frequency interval of length $B>0$, modulation and unitary
dilation reduce the problem to our normalization with $s$ replaced
by the time--bandwidth product $B|E|$. The $L^2(\mathbb{R})$
distance and normalized symmetric difference are unchanged.

\subsection{Localization operators and the proof strategy}

Let $P$ be the orthogonal projection onto $PW$, and define the localization operator $T_E$ as 
\[
T_E=PM_EP\big|_{PW},
\qquad M_Ef=\mathbbm 1_Ef.
\]
% The operator $T_E$ is positive, self-adjoint and compact.
% Indeed, in Fourier coordinates the restriction map from $PW$
% to $L^2(E)$ has kernel $e^{2\pi i x\xi}$ on
% $E\times[-1/2,1/2]$, whose squared integral is $|E|$.
% It is therefore Hilbert--Schmidt, and $T_E$ is its adjoint
% composed with itself. Since $|E|>0$, it is also nonzero.
For a set $E$ with $|E|=s$ we have then
\[
\lambda_E:=\|T_E\|_{\mathrm{op}}
=\max_{\substack{f\in PW\\\|f\|_2=1}}\int_E|f|^2
\le\lambda_0(s).
\]

Let $\lambda_1(s)$ be the second largest concentration eigenvalue on
$I_s$, so that
\[
\lambda_0(s)-\lambda_1(s)>0.
\]
The strict gap follows from classical prolate theory \cite{SP}.
For a general set define the spectral deficit
\[
\eta(E)=\lambda_0(s)-\lambda_E.
\]
Every normalized top eigenfunction $f_E$ satisfies
\[
T_Ef_E=\lambda_Ef_E,
\qquad
\Delta(f_E,E)=\eta(E)\le\Delta(f,E)
\]
for every normalized $f\in PW$.

Our proof starts with the local comparison
\[
\Delta(f,E)\ge
\mathcal D(E)+\frac{\lambda_0(s)-\lambda_1(s)}{2}d_{a_E}(f)^2
-K_s\mathcal D(E)d_{a_E}(f).
\]
% The interval spectral gap controls the function error, while
% fixed measure and optimal centering supply two cancellations
% at the interval endpoints.

We apply this comparison first to a top eigenfunction $f_E$.
Testing with $\psi_s$ gives $\eta(E)\le\mathcal D(E)$, while
completing the square in the local estimate gives
\[
0\le\mathcal D(E)-\eta(E)
\le\frac{K_s^2}{2(\lambda_0(s)-\lambda_1(s))}\mathcal D(E)^2.
\]
A compactness argument using \eqref{eq:equality} shows that
$\eta(E_n)\to0$ forces $\mathcal D(E_n)\to0$. Together, these
facts imply the global bound
$\mathcal D(E)\le C_s\eta(E)$.
Finally, since $\eta(E)\le\Delta(f,E)$, we return to the local
estimate and obtain the main theorem for arbitrary functions.

Section~\ref{sec:local} proves the local comparison.
Section~\ref{sec:global} gives the compactness argument and the
proof of the main theorem, and Section~\ref{sec:sharpness}
discusses the exponent.
Constants denoted by $C_s$ or $c_s$ depend only on $s$ and may
change from line to line; $K_s$ denotes the fixed constant in
Proposition~\ref{prop:local}.

% \section{Preliminaries}
% \label{sec:preliminaries}

% Fourier inversion and Cauchy--Schwarz give, for $g\in PW$,
% \begin{equation}
% \|g^{(j)}\|_\infty
% \le\frac{\pi^j}{\sqrt{2j+1}}\|g\|_2,
% \qquad j=0,1,2,\ldots.
% \label{eq:bernstein}
% \end{equation}
% Indeed, the squared norm of the multiplier $(2\pi i\xi)^j$
% on $[-1/2,1/2]$ is $\pi^{2j}/(2j+1)$.
% Every $g\in PW$ is continuous and tends to zero at infinity,
% since $\widehat g\in L^1$. The formula
% \[
% g(z)=\int_{-1/2}^{1/2}\widehat g(\xi)e^{2\pi iz\xi}\,d\xi
% \]
% also extends $g$ to an entire function. Differentiation under
% the integral is justified uniformly on compact subsets of
% $\mathbb C$. In particular, $|g|^2$ is real analytic on
% $\mathbb R$.

% We also justify the existence of an optimal center used in
% Theorem~\ref{thm:main}.
% The overlap
% \[
% a\longmapsto\int_E|\psi_s(x-a)|^2\,dx
% \]
% is continuous and tends to zero at infinity by dominated
% convergence, since $\psi_s$ is bounded, tends to zero, and
% $|E|<\infty$. It is positive somewhere, because Tonelli's
% theorem gives
% \[
% \int_{\mathbb R}\int_E|\psi_s(x-a)|^2\,dx\,da=s.
% \]
% It therefore attains a positive maximum. In particular,
% $0\le\mathcal D(E)\le\lambda_0(s)$.

\section{The local estimate}
\label{sec:local}

\begin{prop}[Local comparison]
\label{prop:local}
There exists $K_s>0$ such that every normalized $f\in PW$,
every measurable $E$ with $|E|=s$, and every optimal center
$a_E$ satisfy
\begin{equation}
\Delta(f,E)\ge
\mathcal D(E)+\frac{\lambda_0(s)-\lambda_1(s)}{2}d_{a_E}(f)^2
-K_s\mathcal D(E)d_{a_E}(f).
\label{eq:local}
\end{equation}
% This estimate uses neither the global interval inequality
% \eqref{eq:interval} nor its equality characterization.
\end{prop}

\begin{proof}
Translate both $E$ and $f$ so that $a_E=0$, and write
$I=I_s$, $\psi=\psi_s$.
Leaving $\psi$ unchanged, multiply $f$ by a unimodular constant so that
\[
\langle f,\psi\rangle=|\langle f,\psi\rangle|\ge0.
\]
Then $d_0(f)=\|f-\psi\|_2$. 

\medskip\noindent
\textbf{Step 0: decompose the deficit.}
Since $\mathcal D(E) = \lambda_0(s) - \int_E|\psi|^2$ and $\int_I|\psi|^2=\lambda_0(s)$,
\begin{align}
\Delta(f,E)
&= \lambda_0(s)-\int_E|f(x)|^2\,dx \nonumber\\
% &=\mathcal D(E)-\int_E\bigl(|f|^2-|\psi|^2\bigr)\,dx
% \nonumber\\
&=\mathcal D(E)-\int_I\bigl(|f|^2-|\psi|^2\bigr)\,dx +
\int_{\mathbb R}(\mathbbm 1_I-\mathbbm 1_E)
\bigl(|f|^2-|\psi|^2\bigr)\,dx
\nonumber\\
&=\mathcal D(E)
+\left(\lambda_0(s)-\int_I|f|^2\,dx\right)
+
\int_{\mathbb R}(\mathbbm 1_I-\mathbbm 1_E)
\bigl(|f|^2-|\psi|^2\bigr)\,dx. \label{eq:decomposition}
\end{align}
% \begin{align}
% \Delta(f,E)
% ={}&\mathcal D(E)
% +\left(\lambda_0(s)-\int_I|f|^2\right)
% \nonumber\\
% &+\int_{\mathbb R}(\mathbbm 1_I-\mathbbm 1_E)
% \bigl(|f|^2-|\psi|^2\bigr)\,dx.
% \label{eq:decomposition}
% \end{align}
We bound the last two terms separately.

\medskip\noindent
\textbf{Step 1: use the interval spectral gap.}
Write
\[
f=\langle f,\psi\rangle\psi+h,
\qquad h\perp\psi.
\]
The mixed terms vanish because $T_I\psi=\lambda_0(s)\psi$.
The norm of $T_I$ restricted to $\psi^\perp$ is $\lambda_1(s)$, so
\begin{align*}
\int_I|f|^2
&=|\langle f,\psi\rangle|^2\lambda_0(s)
+\langle T_Ih,h\rangle\\
&\le |\langle f,\psi\rangle|^2\lambda_0(s)
+\lambda_1(s)\bigl(1-|\langle f,\psi\rangle|^2\bigr).
\end{align*}
Consequently,
\[
\lambda_0(s)-\int_I|f|^2
\ge(\lambda_0(s)-\lambda_1(s))\bigl(1-|\langle f,\psi\rangle|^2\bigr).
\]
By the choice of phase,
\[
d_0^2(f)=2\bigl(1-|\langle f,\psi\rangle|\bigr).
\]
Using $1-t^2\ge1-t$ for $0\le t\le1$, we obtain
\begin{equation}
\lambda_0(s)-\int_I|f|^2\ge\frac{\lambda_0(s)-\lambda_1(s)}{2}d_0^2(f).
\label{eq:gap}
\end{equation}

\medskip\noindent
\textbf{Step 2: the two cancellations.}
Set
\[
u=\psi^2,\quad v=|f|^2-u.
\]
We use the following properties of $u$:
\begin{equation}
\begin{gathered}
u(s/2)>0,\qquad u'(-s/2)=-u'(s/2)>0,\\
u(x)>u(s/2)\quad (|x|<s/2),\qquad
u(x)<u(s/2)\quad (|x|>s/2).
\end{gathered}
\label{eq:prolate-level}
\end{equation}
These follow from the classical prolate differential equation and
its ground-state solution; see \cite[Section V]{SP}.

In particular, the only zeros of $u-u(s/2)$ are the two simple zeros
$\pm s/2$. Near these points, $|u-u(s/2)|$ is bounded below by a positive
multiple of the distance to the endpoints. Outside fixed neighborhoods
of the endpoints it has a positive lower bound, since it has no other
zeros and $u(x)\to0$ while $u(s/2)>0$. Thus, for some $c_s>0$,
\begin{equation}
|u(x)-u(s/2)|\ge c_s\min\{1,\dist(x,\partial I)\},
\qquad x\in\mathbb R.
\label{eq:weight}
\end{equation}

Equal measures give $\int(\mathbbm 1_I-\mathbbm 1_E)=0$.
Since zero is an optimal center, the function
\[
\Phi(a)=\int_Eu(x-a)\,dx
\]
has a maximum at zero. Boundedness of $u'$ and finiteness
of $|E|$ justify differentiation under the integral, so
$0=\Phi'(0)=-\int_Eu'$.
Also $\int_Iu'=u(s/2)-u(-s/2)=0$, since $u$ is even. Hence
\begin{equation}
\int_{\mathbb R}(\mathbbm 1_I-\mathbbm 1_E)\,dx=0,
\qquad
\int_{\mathbb R}(\mathbbm 1_I-\mathbbm 1_E) u'\,dx=0.
\label{eq:cancellations}
\end{equation}
The signs in \eqref{eq:prolate-level}, together with equal measures,
give, for every prescribed center $a$,
\begin{equation}
D_a(E)=\int_{E\triangle(I_s+a)}|u(x-a)-u(s/2)|\,dx.
\label{eq:weighted}
\end{equation}
Indeed, subtracting $u(s/2)$ from the two integrals
in $D_a(E)$ does not change their difference. This identity also
shows that $D_a(E)=0$ exactly when $E=I_s+a$ up to a null set.
At our optimal center $a_E=0$, it becomes
\begin{equation}
\mathcal D(E)=\int_{\mathbb R}(\mathbbm 1_I-\mathbbm 1_E)(u-u(s/2))\,dx
=\int_{\mathbb R}|\mathbbm 1_I-\mathbbm 1_E|\,|u-u(s/2)|\,dx.
\label{eq:centered-weighted}
\end{equation}
Choose
\[
A=\frac{v(-s/2)+v(s/2)}2,
\qquad
B=\frac{v(-s/2)-v(s/2)}{2u'(-s/2)},
\qquad R=v-A-Bu'.
\]
Since $u'(-s/2)=-u'(s/2)>0$, we have $R(-s/2)=R(s/2)=0$.
The two cancellations give
\[
\int_{\mathbb R}(\mathbbm 1_I-\mathbbm 1_E) v\,dx
=\int_{\mathbb R}(\mathbbm 1_I-\mathbbm 1_E) R\,dx.
\]

\medskip\noindent
\textbf{Step 2.1: bound the remainder.}
As is well known, Fourier inversion and Cauchy--Schwarz give, for $g\in PW$,
\begin{equation}
\|g^{(j)}\|_\infty
\le\frac{\pi^j}{\sqrt{2j+1}}\|g\|_2,
\qquad j=0,1,2,\ldots.
\label{eq:bernstein}
\end{equation}
The identity
\[
v=(f-\psi)\overline f+\psi\,\overline{(f-\psi)}
\]
and \eqref{eq:bernstein} give
\[
\|v\|_\infty
\le\|f-\psi\|_\infty(\|f\|_\infty+\|\psi\|_\infty)
\le2d_0(f).
\]
Differentiating the same identity yields
\begin{align*}
\|v'\|_\infty
&\le\|f'-\psi'\|_\infty(\|f\|_\infty+\|\psi\|_\infty)\\
&\quad+\|f-\psi\|_\infty(\|f'\|_\infty+\|\psi'\|_\infty)
\le\frac{4\pi}{\sqrt3}d_0(f).
\end{align*}
% These estimates apply \eqref{eq:bernstein} to $f$, $\psi$,
% and $f-\psi$, not to $v$, which need not belong to $PW$.
Similarly, we can write $u'=2\psi\psi'$ and $u''=2(\psi')^2+2\psi\psi''$, so
\eqref{eq:bernstein} gives
\[
\|u'\|_\infty\le2\|\psi\|_\infty\|\psi'\|_\infty\le\frac{2\pi}{\sqrt3},
\qquad
\|u''\|_\infty\le2\|\psi'\|_\infty^2+2\|\psi\|_\infty\|\psi''\|_\infty
\le2\pi^2\Bigl(\frac13+\frac1{\sqrt5}\Bigr)\le2\pi^2 .
\]
Since $u'(-s/2)>0$ depends only on $s$,
\[
|A|+|B|\le C_s d_0(f),
\qquad
\|R\|_\infty+\|R'\|_\infty\le C_s d_0(f).
\]
Since $R(-s/2)=R(s/2)=0$, the fundamental theorem of calculus,
together with the sup-norm bound above, gives
\[
|R(x)|\le C_s d_0(f) \min\{1,\dist(x,\partial I)\}
\le C_s d_0(f) \,|u(x)-u(s/2)|,
\]
where the last inequality is \eqref{eq:weight}.
This also bounds $R/(u-u(s/2))$, with continuous extension
at the endpoints.

Finally,
\begin{align*}
\left|\int_{\mathbb R}(\mathbbm 1_I-\mathbbm 1_E) v\,dx\right|
&=\left|\int_{\mathbb R}(\mathbbm 1_I-\mathbbm 1_E) R\,dx\right|\\
&\le C_s d_0(f)\int_{\mathbb R}|(\mathbbm 1_I-\mathbbm 1_E)|\,|u-u(s/2)|\,dx\\
&=C_s\mathcal D(E)d_0(f).
\end{align*}
Fix a constant $K_s>0$ for this estimate.
Combining it with \eqref{eq:decomposition} and \eqref{eq:gap}
proves \eqref{eq:local}.
\end{proof}

\section{From local to global stability}
\label{sec:global}

Since the first stage of the proof concerns the spectral
deficit, we need compactness only for exact maximizing
eigenfunctions. This lets us use their eigenvalue equations
directly.

\begin{lemma}[Qualitative spectral stability]
\label{lem:compactness}
If $|E_n|=s$ and $\eta(E_n)\to0$, then
$\mathcal D(E_n)\to0$.
\end{lemma}

\begin{proof}
Suppose otherwise, and pass to a subsequence on which
$\mathcal D(E_n)\ge\varepsilon_0>0$.
Choose normalized top eigenfunctions $f_n$:
\[
T_{E_n}f_n=\lambda_{E_n}f_n,
\qquad \lambda_{E_n}\longrightarrow\lambda_0(s).
\]
For large $n$, $\int_{E_n}|f_n|^2\ge\lambda_0(s)/2$.
Since $|E_n|=s$, we can translate $E_n$ and $f_n$, and choose
a phase, so that
\[
f_n(0)\ge\sqrt{\lambda_0(s)/(4s)}.
\]
These operations do not change $\eta(E_n)$ or $\mathcal D(E_n)$.

After passing to a further subsequence,
$f_n\rightharpoonup f$ in $PW$.
Point evaluation is continuous, so
$f\ne0$. The same estimates give local uniform convergence:
weak convergence gives convergence at every point, and the
derivative bound gives equicontinuity.

We claim that
\begin{equation}
\int_{E_n}(f_n-f)\overline f\,dx\longrightarrow0.
\label{eq:cross-limit}
\end{equation}
On a fixed bounded interval this follows from local uniform
convergence. On its complement, Cauchy--Schwarz gives the bound
\[
\|f_n-f\|_2\,
\|f\mathbbm 1_{\{|x|>R\}}\|_2,
\]
which is uniformly small for large $R$.

Testing the eigenvalue equation against the fixed function $f$
now gives
\[
\int_{E_n}|f|^2\,dx
=\lambda_{E_n}\langle f_n,f\rangle+o(1)
\longrightarrow\lambda_0(s)\|f\|_2^2.
\]
In particular,
\begin{equation}
\lambda_0(s)\|f\|_2^2-\int_{E_n}|f|^2\,dx
\longrightarrow0.
\label{eq:profile}
\end{equation}

Let $A$ be a superlevel set of $|f|^2$ with $|A|=s$. Let $t>0$ be such that $A=\{x:|f(x)|^2>t\}$.
For any $B$ of measure $s$, equal measures give
\[
\int_A|f|^2-\int_B|f|^2
=\int_{A\setminus B}(|f|^2-t)
+\int_{B\setminus A}(t-|f|^2)\ge0.
\]
Therefore \eqref{eq:profile} and \eqref{eq:interval} imply
\[
\lambda_0(s)\|f\|_2^2
=\lim_n\int_{E_n}|f|^2
\le\int_A|f|^2
\le\lambda_0(s)\|f\|_2^2.
\]
The equality characterization \eqref{eq:equality} gives
$f=c\psi_s(\cdot-b)$ for some $c\ne0$ and $b\in\mathbb R$.
Testing the definition of $\mathcal D(E_n)$ at this center,
\[
0\le\mathcal D(E_n)
\le
\frac{\lambda_0(s)\|f\|_2^2-\int_{E_n}|f|^2}{\|f\|_2^2}
\longrightarrow0,
\]
contradicting $\mathcal D(E_n)\ge\varepsilon_0$.
\end{proof}

\begin{proof}[Proof of Theorem~\ref{thm:main}]
We first control the set discrepancy by the spectral deficit,
and then return to an arbitrary normalized function.

\medskip\noindent
\textbf{Step 1: compare the two deficits.}
Fix $E$, an optimal center $a_E$, and a normalized top
eigenfunction $f_E$.
Testing the Rayleigh quotient with $\psi_s(\cdot-a_E)$ gives
$\eta(E)\le \mathcal D(E)$. Proposition~\ref{prop:local}, applied to $f_E$,
and completion of the square yield
\begin{equation}
\begin{aligned}
0\le \mathcal D(E)-\eta(E)
&\le K_s\mathcal D(E)d_{a_E}(f_E)-\frac{\lambda_0(s)-\lambda_1(s)}{2}d_{a_E}^2(f_E)\\
&\le A_s\mathcal D(E)^2,
\end{aligned}
\label{eq:spectral-quadratic}
\end{equation}
where 
\begin{equation*}
    A_s=\frac{K_s^2}{2(\lambda_0(s)-\lambda_1(s))}>0.
\end{equation*}

\medskip\noindent
\textbf{Step 2: make the spectral comparison global.}
We claim that there is $S_s>0$ such that
\begin{equation}
\mathcal D(E)\le S_s\eta(E)
\qquad\text{for every }|E|=s.
\label{eq:global-spectral}
\end{equation}
Suppose not. There would be sets $E_n$ with
\[
\mathcal D(E_n)>n\eta(E_n).
\]
Since $\mathcal D(E_n)\le\lambda_0(s)$, we have
$\eta(E_n)\to0$. Lemma~\ref{lem:compactness} gives
$\mathcal D(E_n)\to0$. For large $n$,
$A_s\mathcal D(E_n)\le1/2$, and
\eqref{eq:spectral-quadratic} then implies
\[
\mathcal D(E_n)
\le\eta(E_n)+A_s\mathcal D(E_n)^2
\le\eta(E_n)+\tfrac12\mathcal D(E_n).
\]
Hence $\mathcal D(E_n)\le2\eta(E_n)$, a contradiction.
This proves \eqref{eq:global-spectral}.

\medskip\noindent
\textbf{Step 3: return to arbitrary functions.}
Let $\|f\|_2=1$.
Since $\eta(E)\le\Delta(f,E)$, \eqref{eq:global-spectral} gives
\[
\mathcal D(E)\le S_s\Delta(f,E).
\]
Moreover, Proposition~\ref{prop:local} implies
\[
\frac{\lambda_0(s)-\lambda_1(s)}{2}d_{a_E}^2(f)
\le\Delta(f,E)-\mathcal D(E)+K_s\mathcal D(E)d_{a_E}(f).
\]
Phase minimization gives
\[
d_{a_E}^2(f)=2\bigl(1-|\langle f,\psi_s(\cdot-a_E)\rangle|\bigr)\le2.
\]
Consequently,
\[
\frac{\lambda_0(s)-\lambda_1(s)}{2}d_{a_E}^2(f)
\le\Delta(f,E)+\sqrt2K_s\mathcal D(E)
\le(1+\sqrt2K_sS_s)\Delta(f,E).
\]
Together with the set estimate, this proves \eqref{eq:main}.

\medskip\noindent
\textbf{Step 4: compare with symmetric difference.}
For any prescribed center $a$, we show that
\begin{equation}
|E\triangle(I_s+a)|^2\le C_sD_a(E).
\label{eq:asymmetry}
\end{equation}
Translate to $a=0$ and put $\alpha=|E\triangle I_s|\le2s$.
The case $\alpha=0$ is immediate. Otherwise choose
\[
t=\frac{\alpha}{8(1+s)}<1.
\]
At most $4t\le\alpha/2$ of the symmetric difference lies
within distance $t$ of the two endpoints. Thus at least
$\alpha/2$ lies at distance at least $t$.
By \eqref{eq:weight} and \eqref{eq:weighted},
\[
D_0(E)\ge c_s t\,\frac\alpha2
=\frac{c_s}{16(1+s)}\alpha^2.
\]
This proves \eqref{eq:asymmetry}. Applying it at $a_E$ and
using \eqref{eq:main} proves \eqref{eq:main-measure}.
\end{proof}

The equality characterization was used only in the compactness
lemma to identify a nonzero weak limit. We did not need strong
convergence of the maximizing functions or convergence of the
sets first. The quadratic comparison \eqref{eq:spectral-quadratic}
follows directly from Proposition~\ref{prop:local}; compactness
makes it global \eqref{eq:global-spectral}.
The compactness argument does not give an explicit value of
$S_s$, and the final stability constant remains non-explicit.

\section{The square-root exponent}
\label{sec:sharpness}

Let $\psi_{1,s}$ be a real, normalized second prolate for $I_s$,
so that
\[
T_{I_s}\psi_{1,s}=\lambda_1(s)\psi_{1,s},
\qquad \langle\psi_{1,s},\psi_s\rangle=0.
\]
Set
\[
f_t=\sqrt{1-t^2}\,\psi_s+it\psi_{1,s},
\qquad 0<t<1.
\]
Orthogonality and the eigenvalue equations give
\[
\|f_t\|_2=1,
\qquad
\Delta(f_t,I_s)=(\lambda_0(s)-\lambda_1(s)) t^2.
\]
At the same time,
\[
d_0(f_t)^2=2(1-\sqrt{1-t^2})\sim t^2
\qquad(t\to0).
\]
Equation \eqref{eq:weighted} shows that zero is the unique
optimal set center of $I_s$. Therefore no estimate
$d_{a_E}(f)\le C_s\Delta(f,E)^\beta$ with $\beta>1/2$
can hold for all normalized $f$ and all sets of measure $s$.
This proves optimality of the square-root exponent in Theorem \ref{thm:main}.

\section*{Acknowledgments} The idea and proof included in Section \ref{STFT} were developed during my research visit from February to May 2026 at the Universitat de Barcelona, which was supported by Grant 361374 of the Research Council of Norway. I thank the Departament de Matemàtiques i Informàtica and Joaquim Ortega-Cerdà, who supervised the visit, for their warm hospitality. The idea to use similar techniques in the setting of Paley--Wiener spaces, which led to Proposition \ref{prop:local}, came to me during Joaquim Ortega-Cerdà's lecture series during \textsl{Complex Analysis and Operator Theory: a Series of 3-Hour Lectures}, organized by Dipartimento di Matematica "Federigo Enriques", Università degli Studi di Milano. 

I acknowledge the use of ChatGPT and Claude in finishing this manuscript. AI was used to assist me getting from local to global stability as well as in the writing process of the whole manuscript.

\end{document}